\documentclass[12pt]{amsart}

\theoremstyle{definition}
\newtheorem{definition}{Definition}
\theoremstyle{plain}
\newtheorem{theorem}{Theorem}
\newtheorem{corollary}{Corollary}
\newtheorem{lem}{Lemma}
\newtheorem{prop}{Proposition}

\theoremstyle{remark}

\newtheorem{remark}{Remark}
\DeclareMathOperator{\CC}{\mathbb C}

\DeclareMathOperator{\ZZ}{\mathbb Z}
\DeclareMathOperator{\NN}{\mathbb N}

\DeclareMathOperator{\dis}{dis}
\DeclareMathOperator{\interior}{int}
\title{A characterization of polynomial density on curves via matrix algebra.}

\address{ Departamento de Matem\'atica Aplicada, Facultad de Inform\'atica de Madrid\\
        Universidad Polit\'ec\-ni\-ca, Campus de Montegancedo\\
      Boadilla del Monte, 28660 Madrid Spain, Phone: +34913367429 \\
      }

\email{cescribano@fi.upm.es}
\email{rngonzalo@fi.upm.es}
\email{emilio@fi.upm.es}

\begin{document}

\author{Escribano, C.}
\author{Gonzalo, R.}
\author{Torrano, E.}

\begin{abstract} In this work, our aim is to obtain conditions to assure polynomial approximation in Hilbert spaces  $L^{2}(\mu)$, with $\mu$ a compactly supported measure in the complex plane,  in terms of properties of the associated moment matrix to the measure $\mu$. In order to do it, in the more general context of Hermitian positive semidefinite matrices we introduce three indexes $\gamma(\mathbf{M})$, $\lambda(\mathbf{M})$ and $\alpha(\mathbf{M})$  associated with different optimization problems concerning theses matrices.  Our main result is a characte\-rization of density of polynomials in the case of measures supported on Jordan curves with non empty interior using the index $\gamma$ and other specific index related to it. Moreover, we provide a new point of view of bounded point evaluations associated to a measure in terms of the index $\gamma$ that will allow us to give an alternative proof of Thomson's theorem  in \cite{Brennan} by using these matrix indexes. We point out that our techniques are based in matrix algebra  tools in the frame of Hermitian Positive Definite matrices and in the computation of certain indexes related to some  optimization problems for infinite matrices.
\end{abstract}

\maketitle
\begin{quotation} {\sc {\footnotesize Keywords}}. {\small Hermitian moment problem,  orthogonal polynomials,
smallest eigenvalue, measures, polynomial density.}
\end{quotation}

\section{Introduction}

\noindent Thorough the paper we consider positive  Borel measures $\mu$ which are finite and  compactly supported in the complex plane. We always consider nontrivial measures, that is, measures with an infinite amount of points in their support. The problem of completeness of polynomials in the Hilbert space $L^{2}(\mu)$ is the following: for a certain measure $\mu$, are polynomials dense in the space $L^{2}(\mu)$? In other words, denote by $P^{2}(\mu)$ the closure of the polynomials  in the space  $L^{2}(\mu)$, the question is under what conditions  the equality $L^{2}(\mu)=P^{2}(\mu)$ is true. In the particular case of $\mu$ being the two-dimensional Lebesgue  measure on  an arbitrary domain $G$  and $L^{2}(G)$ the associated Hilbert space,  the classical results of approximation by polynomials can be seen in e.g.  \cite{Gaier}, where it is explored the question of which assumptions on $G$ will be assumed in order to have polynomials density in $L^{2}(G)$. The related questions about the existence of approximation rational, entire, or meromorphic were solved by the great theorem of Mergelyan in 1951 which completes a long chain of theorems about approximation by polynomials.

\medskip
\noindent The problem of density of polynomials is also an interest topic in the Theory of orthogonal polynomials associated with a measure. Indeed, in the particular case of orthogonal polynomials in the unit circle the well known Szeg\"{o} theory ( see e.g. \cite{szego},\cite{Simon1}) deals with  the problem of polynomial approximation  using proper tools of orthogonal polynomials.

\medskip
\noindent On the other hand, in \cite{EGT1} a necessary condition was provided to assure polyn\-omial approximation using the behaviour of the smallest eigenvalues  of the finite sections of the moment matrix associated to a measure.Along this work we  follow this matrix approach in order to obtain the main results in this paper.

\medskip

\noindent Thorough this paper we consider infinite positive definite hermitian matrices $\mathbf{M}=(c_{i,j})_{i,j=0}^{\infty}$.
As in \cite{Duran}, \cite{EGT1} an
Hermitian positive definite  matrix (in short, an HPD matrix) defines an inner product $\langle \;,
\; \rangle$ in the space $\mathbb{P}[z]$ of all polynomials with complex coefficientes in the following way:
if  $p(z)=\sum_{k=0}^{n}v_kz^k$ y
$q(z)=\sum_{k=0}^{m}w_kz^k$ then,
$$
\langle p(z),q(z) \rangle=v\mathbf{M}w^{*},
$$

\noindent being $v=(v_0,\dots,v_n,0,0, \dots), w=(w_0,\dots,w_m,0,0, \dots) \in c_{00}$,  where $c_{00}$ is the space
of all complex sequences with only finitely many non-zero entries. The associated norm is $\Vert p(z) \Vert^2 = \langle p(z),p(z)\rangle $ for every $p(z)\in \mathbb{P}[z]$.

\noindent An interesting   class of HPD matrices are those which are  moment matrices
 with respect to a measure $\mu$, i.e., HPD matrices  $\mathbf{M}=(c_{i,j})_{i,j=0}^{\infty}$ such that there exists a representating measure $\mu$ with infinite support on $\CC$ and  finite moments for all $i,j\geq 0$,
$$
c_{i,j}= \int z^{i} \overline{z}^jd\mu.
$$

\noindent Our aim here is to obtain conditions to assure polynomial approximation in Hilbert spaces  $L^{2}(\mu)$, with $\mu$ a compactly supported measure in the complex plane,  in terms of properties of the associated matrix  $\mathbf{M}$. In order to do it, in the more general context of Hermitian positive {\it semidefinite} matrices we introduce three matrix indexes $\gamma(\mathbf{M})$, $\lambda(\mathbf{M})$ and $\alpha(\mathbf{M})$, each one related with different  optimization matrix problems. Among these indexes we highlight the index $\gamma$ that as we have realized will be essential to characterize the polynomial density in our context. The other index $\lambda$ is related to the asymptotic behaviour of the smallest eigenvalues in our previous works (see \cite{EGT1}). These indexes will be introduced in
the first section some properties of them will be given. We also provide an application to the  index $\lambda$ to some problems of perturbations of measures in the same direction as in \cite{Marcellan}.

\medskip

\noindent In the following section we consider the case when the Hermitian semidefinite positive matrices are moment matrices associated with a measure $\mu$ with compact support in the complex plane. Our main result is a characterization of completeness of polynomials in the associated space $L^{2}(\mu)$, in the case of Jordan curves with $0$ in its interior, in terms of the index $\gamma$ of the moment matrix associated with the measure $\mu$.

\medskip

\noindent In the last section we give our main result which is characterization of density of polynomials on Jordan curves with non empty interior in terms of an specific index related to the index $\gamma$. Moreover, we provide a matrix algebra point of view of the notion of bounded point evaluation of a measure. This will lead us to obtain a new proof of Thomson's theorem    in \cite{Brennan}, in the particular context of Jordan curves with non empty interior,  using our techniques and our results.

\medskip

\noindent Finally, we point out that our approach is based in matrix algebra tools in the frame of general HPD and in the computation of certain indexes related to some optimization problems for infinite matrices. This point of view would allow solving certain matrix optimization problems in terms of the Theory of orthogonal polynomials and on the other hand would let obtaining results of interest concerning orthogonal polynomials using the matrix optimization tools.

\noindent
\section{New indices of an HPD matrix and connections with the polynomial approximation }
\noindent In this section we introduce some indices associated with general Hermitian semi-positive definite matrices. Let $\mathbf{M}=(c_{i,j})_{i,j=0}^{\infty}$ be an infinite Hermitian matrix, i.e. $c_{i,j}=\overline{c_{j,i}}$. We say that an infinite  Hermitian matrix  $\mathbf{M}$ is {\it positive definite} if (in short an HPD matrix) if $\vert \mathbf{M}_n \vert >0$ for all $n\geq 0$, where $\mathbf{M}_n$ is the truncated matrix of size $(n+1) \times (n+1)$ of $\mathbf{M}$. In an analogous way, if  $\vert \mathbf{M}_n \vert \geq 0$ for all $n\geq 0$ we say that $\mathbf{M}$ is an Hermitian  semi-positive definite matrix (in short HSPD). In the sequel we use the same notation  as in \cite{EGT1}, we denote by $(1,v) \equiv (1,v_1,\dots,v_n,0,0,)$ for every $v=(v_1,v_2,\dots)\in c_{00}$ and by $(v,1,0,\dots) \equiv (v_0,\dots,v_{n-1},1,0,\dots)$
\begin{definition} Let $\mathbf{M}$ be an infinite Hermitian semi-definite positive matrix. We define
$$
\gamma(\mathbf{M})=
\inf  \{ (1,v)\mathbf{M}\left( \!\!\!
\begin{array}{c}
                   1 \\
                   v^{*} \\
                 \end{array} \!\!\!
               \right   ),
               v\in c_{00} \}.
$$
\noindent This index always exists and $\gamma(\mathbf{M})\geq 0$.
\end{definition}
\begin{definition} Let $\mathbf{M}$ be an
infinite Hermitian semi-definite
positive matrix. We define
$$
\lambda(\mathbf{M})= \inf \{v\mathbf{M}v^{*}; vv^{*}=1, v\in c_{00}\}.
$$
\noindent This index always exists and $\lambda(\mathbf{M})\geq 0$.
\end{definition}
\begin{remark}  Note that there is an important link between eigenvalue problems and optimization is the Rayleigh quotients. Indeed, for Hermitian matrices $\mathbf{M}_n$ it is well known that if we define $Q_n(v)=\dfrac{v\mathbf{M}_nv^{*}}{vv^{*}}$ for $0\neq v\in \CC^n$ it is well known that $\min \{ Q_n(v) \, v\in \CC^n\}$ or $\max \{Q_n(v) : v\in \CC^n\}$ gives the extreme eigenvalues of $\mathbf{M}_n$. We denote by $\lambda_n$ the smallest eigenvalue of $\mathbf{M}_n$ as in \cite{EGT1}; that is if $\Vert \cdot \Vert_2$ is the euclidean norm in $\CC^{n+1}$
$$
\lambda_{n}=\inf \{v\mathbf{M}_nv^{*}; v\in \CC^{n+1}, \Vert v \Vert_2=1\}
$$
\noindent Moreover, the sequence $\{\lambda_n\}_{n=0}^{\infty}$ is an non non decreasing sequence and
$$
\lambda(\mathbf{M}) = \lim_{n\to \infty} \lambda_n.
$$
\end{remark}
\begin{definition}  Let $\mathbf{M}$ be an infinite Hermitian semi-definite positive matrix. We define $\alpha(\mathbf{M})$
$$
\alpha(\mathbf{M})= \inf  \{ (v,1)\mathbf{M}\left( \!\!\!
                 \begin{array}{c}
                    v^{*}\\
                  1\\
                 \end{array} \!\!\!
               \right) \, ,
               v\in c_{00} \}.
$$
\noindent This index always exists and $\alpha(\mathbf{M})\geq 0$.
\end{definition}
\noindent Next we relate these three indexes.
\begin{prop} Let $\mathbf{M}$ be an infinite Hermitian semi-definite positive matrix. Then,
\begin{enumerate}
\item[i)] $\lambda(\mathbf{M}) \leq \gamma(\mathbf{M})$
\item[ii)] $\lambda(\mathbf{M}) \leq \alpha(\mathbf{M})$
\end{enumerate}
\end{prop}

\begin{proof} We first show $i)$. Let $v\neq 0$ and consider the normalized vector $\dfrac{v}{(vv^{*})^{1/2}}$. By the definition of $\lambda(\mathbf{M})$ we have that
$$
\dfrac{v}{(vv^{*})^{1/2}}\mathbf{M}\dfrac{v^{*}}{(vv^{*})^{1/2}} \geq \lambda(\mathbf{M}).
$$
\noindent By taking in mind that for any $t>0$ it holds
  $(tv)\mathbf{M}(tv)^{*}=t^{2}v\mathbf{M}v^{*}$, then it
easily follows that
$$
v\mathbf{M}v^{*} \geq \lambda(\mathbf{M}) vv^{*}.
$$
\noindent Consequently,
$$
(1,v)\mathbf{M}\left(  \!\!
                 \begin{array}{c}
                   1 \\
                   v^{*} \\
                 \end{array}  \!\!
               \right) \geq
               \lambda(\mathbf{M})(1,v)(1,v)^{*} \geq \lambda(\mathbf{M}).
               $$
              \noindent By taking
              the infimum we obtain
              $\lambda(\mathbf{M}) \leq \gamma(\mathbf{M}).$
\noindent In the same way, we have  that for every $v\in c_{00}$,
   $$ (v,1,0,\dots)\mathbf{M}\left( \!\!
                 \begin{array}{c}
                    v^{*}\\
                  1\\
                  0\\
                  \vdots
                 \end{array} \!\!
               \right) \geq \lambda(\mathbf{M}) (v,1,0,\dots)(v,1,0,\dots)^{*}
                                 \geq \lambda(\mathbf{M})$$
\noindent and consequently $\lambda(\mathbf{M}) \leq \alpha(\mathbf{M})$.
\end{proof}
\begin{remark} The equality is not true in general, even for Toeplitz matrices. For instance, consider the matrix
 \[ \mathbf{T} =\left ( \begin {array}{cccccc} 1&-1/2&1/4&-1/8&1/16& \ldots
\\ \noalign{\medskip}-1/2&1&-1/2&1/4&-1/8& \ldots \\ \noalign{\medskip}1/4&-1/2
&1&-1/2&1/4& \ldots \\ \noalign{\medskip}-1/8&1/4&-1/2&1&-1/2
& \ldots \\ \noalign{\medskip}1/16&-1/8&1/4&-1/2&1& \ldots\\
\vdots & \vdots & \vdots & \vdots & \vdots & \ddots
\end {array} \right ).
\]
By induction we obtain that $|\mathbf{T}_{n}|/|\mathbf{T}_{n-1}|=3^{n}/4^{n}$, consequently
$\kappa^2 = \lim_{n}|\mathbf{T}_{n-1}|/|\mathbf{T}_{n}|= 4/3$, where $\kappa$ is the  leading coefficient of orthonormal polynomial sequence. We get the associated weight $w(t)$ of the Toeplitz matrix $\mathbf{T}$,
 \begin{multline*}
 \qquad \qquad w(t)= \sum_{n=-\infty}^{\infty} \frac{(-1)^n}{2^{|n|}} z^{n} = \frac{3z}{(z+2)(2z+1)}=\\ \frac{3 e^{it}}{(2+e^{it})(2e^{it}+1)}= \frac{3}{5+4\cos(t)}
 \end{multline*}
It follows by \cite{grenander}
 \[\lambda(\mathbf{T})= \min_{t \in [0,2\pi]} w(t)=  \frac{1}{3}.\]
Therefore by \cite{EGT1} it can be obtained
\[ \gamma(\mathbf{T})=\dfrac{1}{\kappa^2}=\frac{3}{4}.\]
\end{remark}
\begin{remark}
Using the results in  \cite{EGT1} it can be easily obtained that
in the case of Toeplitz matrices $\mathbf{T}$ the indexes
$\gamma(\mathbf{T}),\alpha(\mathbf{T})$ coincide, that is:
\end{remark}
\begin{prop} \cite{EGT1} Let $\mathbf{T}$ be an infinite Toeplitz
SHPD matrix. Then,
$$
\gamma(\mathbf{T})= \alpha(\mathbf{T}).
$$
\end{prop}
\bigskip

\noindent In the set
   of Hermitian semi-definite positive infinite matrices we may define an order in the following way: we say that $\mathbf{M}_1 \leq \mathbf{M}_2$ if $v\mathbf{M}_1v^{*} \leq v\mathbf{M}_2v^{*}$ for every $v\in c_{00}$. We have the following results:
\begin{lem} \label{lema1}
Let $\mathbf{M}_1,\mathbf{M}_2$ be infinite Hermitian semi-definite positive matrices
with $\mathbf{M}_1 \leq \mathbf{M}_2$ then:
\begin{enumerate}
\item $\lambda(\mathbf{M}_1) \leq \lambda(\mathbf{M}_2)$.
\item $\gamma(\mathbf{M}_1) \leq \gamma(\mathbf{M}_2)$.
\item $\alpha(\mathbf{M}_1) \leq \alpha(\mathbf{M}_2)$.
\end{enumerate}
\end{lem}

\bigskip

\noindent We give some applications of the above result to some perturbation results in the same lines as \cite{Marcellan}. Let $\sigma$ be a non trivial positive measure with support in $\mathbb{T}$, in \cite{Marcellan} it is obtained that if the measure $\sigma$ verifies Szeg\"{o} condition and $\widetilde{\sigma}$ is the perturbed measure of $\sigma$ by  the  normalized Lebesgue measure in the unit circle, that is  $d\widetilde{\sigma}=d\sigma + r\dfrac{d\theta}{2\pi}$ for  $r>0$, then $\widetilde{\sigma}$ also verifies Szeg\"{o} condition. Using our techniques we generalize this result pointing out that there is no need  to require that $\sigma$ verifies Szeg\"{o} condition since the conclusion is true always. Indeed we have:

\begin{corollary} \label{corolario1}
Let $\sigma_1,\sigma_2$   positive  positive measures with support on $\mathbb{T}$. Assume that  one of them verifies Szeg\"{o} condition,  then the measure  $\sigma:=\sigma_1+\sigma_2$  verifies Szeg\"{o} condition. In particular,
if $d\widetilde{\sigma}=d\sigma + r\dfrac{d\theta}{2\pi}$ with $r>0$ for some  positive  measure with support in $\mathbb{T}$ then  $\widetilde{\sigma}$ verifies Szeg\"{o} condition.
\end{corollary}

\begin{proof} Let  $\mathbf{T}_1,\mathbf{T}_2,\mathbf{T}_{\sigma}$ be the Toeplitz semi-definite  positive moment matrices associated with $\sigma_1,\sigma_2,\sigma$. Assume that $\sigma_1$ verifies Szeg\"{o} condition, then $\gamma(\mathbf{T}_{1})>0$. By Lemma \ref{lema1}  it follows that $\gamma(\mathbf{T}_{\sigma})\geq \gamma(\mathbf{T}_{1})>0$ and consequently $\sigma$ verifies Szeg\"{o} condition.

\begin{remark} Note that in Corollary \ref{corolario1} it is not required that both measures are non trivial; indeed, we may consider a perturbation by a finite amount of atomic points.
\end{remark}

\noindent In the particular case of $d\widetilde{\sigma}=d\sigma + r\dfrac{d\theta}{2\pi}$, obviously the normalized Lebesgue measure in the unit circle verifies Szego condition and consequently $\widetilde{\sigma}$ also verifies Szeg\"{o} condition independently of $\sigma$.
\end{proof}

\noindent From now on we  consider now an infinite HPD matrix $\mathbf{M}$. This matrix induces an inner product in the vector space $c_{00}$. In this way the space $c_{00}$ endowed with such a norm is a prehilbert space with the prehilbertian norm
 $$
 \Vert v \Vert^2_{\mathbf{M}}=v\mathbf{M}v^{*}.
 $$
\noindent  We consider the completion of this space with such norm that we denote by $P^{2}(\mathbf{M})$; we may apply Gram-Schmidt orthogonalization procedure to the canonical algebraic basis $\{e_n\}_{n=0}^{\infty}$ in $c_{00}$ and  we obtain the unique orthonomal basis $\{v_{i} \}_{i=0}^{\infty}$ with $v_i=(v_{0,i},\dots,v_{i,i},0,\dots)$ for $i\in \NN_{0}$ and $v_{i,i}>0$. Consider $w_{n}=\dfrac{v_n}{\Vert w_n \Vert_{\mathbf{M}}}$ the orthogonal monic vector. It is clear that $\Vert w_n \Vert^{2}=\dfrac{1}{\vert v_{n,n}\vert^2}$ for every
$n\in \NN_{0}$. We denote by $\overline{[e_1,e_2,\dots]}^{\mathbf{M}}$ the closed vector subspace generated by the set of vectors $e_n's$ with the norm induced by the matrix $\mathbf{M}$.

\begin{prop} \label{proposicion3}
Let $\mathbf{M}$ be an Hermitian definite positive matrix. Let $\{e_n\}_{n=0}^{\infty}$ be the canonical basic sequence in $c_{00}$, then
\begin{enumerate}
\item $\gamma(\mathbf{M})={ \dis}^2(e_{0},\overline{[e_1,e_2,\dots]}^{\mathbf{M}})$.
\item $\alpha(\mathbf{M})=\inf \{ {\it \dis}^{2}(e_n,[e_0,\dots, e_{n-1}]), n\in \NN\}$.
\end{enumerate}
\end{prop}
\noindent From the results in \cite{EGT1} we have the following infinite dimensional version of the result in
 the case of Hermitian positive definite matrices.
\begin{prop} \label{proposicion4}
Let $\mathbf{M}$ be an HSPD matrix and let $\{v_{0},v_1,v_{2},\dots\}$ be the orthonormal basis in $P^{2}(\mathbf{M})$ with respect the inner product induced by $\mathbf{M}$ with $v_i=(v_{0,i},\dots,v_{i,i},0,\dots)$ for $i\in \NN_{0}$ and $v_{i,i}>0$. Then,

$$ \gamma(\mathbf{M})=\dfrac{1}{\sum_{i=0}^{\infty}\vert v_{0,i}\vert^2}$$
\noindent where the left side is zero if  $\sum_{i=0}^{\infty}\vert v_{0,i}\vert^2=\infty$.

\end{prop}
\section{When the Hermitian semi-definite positive matrices are moment matrices.}
\noindent In this section we consider the most important example of Hermitian definite positive matrices which are the moment matrices with respect a Borel non trivial compactly supported measure $\mu$ in the complex plane $\mathbf{M}(\mu)$. In this case, the space $c_{00}$ is replaced by the space of polynomials $\mathbb{P}[z]$ via the identification
$$
v=(v_0,\dots, v_n,0,0,\dots) \equiv p(z)= v_{0}+v_1z+\dots +v_nz^n.
$$
\noindent The associated norm in $\mathbb{P}[z]$ with respect to $\mathbf{M}:=\mathbf{M}(\mu)$ is the usual norm of the polynomials  in the space $L^{2}(\mu)$; that is for every $p(z)\in \mathbb{P}[z]$.
$$
\Vert p(z) \Vert^{2}_{P^{2}(\mathbf{M})}= \int \vert p(z) \vert^2 d\mu.
$$
\noindent As usual the completion of the space of polynomials in the space $L^{2}(\mu)$ is  denoted by $P^{2}(\mu)$,  $\{\varphi_n(z)\}_{n=0}^{\infty}$ is the sequence of orthonormal polynomials and $\{\Phi_n(z)\}_{n=0}^{\infty}$ is the associated sequence of monic orthogonal polynomials. We denote by $P^{2}_0(\mu)$ the completion of polynomials vanishing at zero.  The well known extremal properties of the monic polynomials and the $n$-kernels are just obtained by reformulating in this context  Proposition \ref{proposicion4} above, which, as we have pointed out,  are results obtained by algebraical proofs in the more general  context of the general Hermitian definite matrices. Indeed, reformulating Lemma \ref{lema1}
we obtain:
\begin{lem} \label{lema2}
Let $\mu$ be a measure compactly supported measure with infinite support in the complex plane and let $\mathbf{M}:=\mathbf{M}(\mu)$ be the associated moment matrix. Then,
$$
\gamma(\mathbf{M})=\dis \left(\mathbf{1},P^{2}_0(\mu)\right).
$$
\end{lem}
\noindent We need the following lemma:
\begin{lem} \label{lema3}
Let $\mu$ be a non trivial positive compactly supported measure in  $\CC$ with $0\notin {\it supp}(\mu)$. The following are equivalent,
\begin{enumerate}
\item $\gamma(\mathbf{M})=0$.
\item For all $k\in \ZZ$,  $z^{k} \in \overline{[z^{k+1},z^{k+2},\dots]}^{L^{2}(\mu)}$.
\end{enumerate}
\end{lem}
\begin{proof} First of all there exist $R>0$ and $\alpha >0$ such that $\alpha \leq \vert z \vert \leq R$ for every
$z\in {\it supp}(\mu)$. Consequently, for every $k \in \ZZ$ and for every $v_0,v_1,\dots, v_n\in \CC, n\in \NN$ it follows
$$
\alpha^{2k} \int \vert {\mathbf{1}}-v_0-v_1z-\dots -v_nz^n \vert^2 d\mu \leq
\int \vert z \vert^{2k}\vert \mathbf{1}-v_0-v_1z-\dots -v_nz^n \vert^2 d\mu =
$$
$$
\int \vert z^{k} -v_0z^{k+1}-v_1z^{k+1}-\dots -v_nz^{k+n} \vert^2 d\mu
\leq R^{2k} \int
\vert {\mathbf{1}}-v_0-v_1z-\dots -v_nz^n \vert^2 d\mu$$

\noindent Therefore, $\mathbf{1}\in P^{2}_{0}(\mu)$ if and only if $z^{k} \in \overline{[z^{k+1},z^{k+2},\dots]}^{L^{2}(\mu)}$.
\end{proof}

\bigskip

\noindent As a consequence, for
 compactly supported measures with $0 \notin {\it supp}(\mu)$ the condition $\gamma(\mathbf{M})$ characterizes completeness of polynomials in the closed subspace of Laurent polynomials in $L^{2}(\mu)$ denoted by $\CC[z,z^{-1}]=\overline{[1,z,\frac{1}{z},z^2,\frac{1}{z^2},\dots]}^{L^{2}(\mu)}$:

\begin{corollary} \label{corolario2}
Let $\mu$ be a non trivial positive compactly supported measure in  $\CC$ with $0\notin {\it supp}(\mu)$. The following are equivalent,
\begin{enumerate}
\item $\gamma(\mathbf{M})=0$.
\item $P^{2}(\mu)=\CC[z,z^{-1}]$.
\end{enumerate}
\end{corollary}
\noindent In particular, for non trivial positive  measures $\sigma$  supported
 in the unit circle it is well known that  Laurent polynomials are dense in $L^{2}(\sigma)$ and therefore the  condition $\gamma(\mathbf{T})=0$ characterizes completeness of polynomials in $L^{2}(\sigma)$. More generally this result will be true whenever Laurent for measures $\mu$ such that  polynomials are dense in $L^{2}(\mu)$. Moreover, we have:
\begin{theorem} \label{teorema1}
Let $\Gamma$ be a Jordan curve such that $0 \in  \interior \Gamma$ and let $\mu$ be a measure with support in $\Gamma$ and associated moment matrix $\mathbf{M}$. The following are equivalent:
\begin{enumerate}
\item $\gamma(\mathbf{M})=0$.
\item $P^{2}(\mu)=L^{2}(\mu)$.
\end{enumerate}
\end{theorem}
\begin{proof} In  \cite{Gaier}   the following consequence of Mergelyan's theorem is given: if $\Gamma$ is a Jordan curve, $0\in  \interior \Gamma$, $f$ is continuous on $\Gamma$, then for every $\epsilon >0$ there exists a $P(z)=\sum_{n=-N}^{N} a_nz^n$ such that $\vert f(z)-P(z)\vert < \epsilon$ for every $z \in \Gamma$. This means that
$\CC[z,z^{-1}]$ is dense in the space of continuous functions on $\Gamma$ with the uniform norm, that is, for every $f$ continuous in $\Gamma$ and $\epsilon >0$ there exists  $g(z)\in \CC[z,z^{-1}]$ such that for every $z\in \Gamma$
$$
\vert f(z) -g(z)\vert \leq \epsilon.
$$
\noindent Therefore,
$$
\int \vert f(z) -g(z)\vert^2 \leq \epsilon^2 \mu(\Gamma)
$$
\noindent and consequently $\CC[z,z^{-1}]$ is dense in the space of continuous functions in $L^{2}(\mu)$. Since for compactly supported measures continuous functions are dense in $L^{2}(\mu)$ we obtain that
$P^{2}(\mu)=\CC[z,z^{-1}]=L^{2}(\mu)$ if and only if $\gamma(\mathbf{M})=0$ as we required.
\end{proof}
\noindent As a consequence of the above results we have the well known consequence of Szeg\"{o} theorem for measures supported in the unit circle:
\begin{corollary}  \label{corolario3}
Let $\sigma$ be a non trivial positive measure with support in $\CC$ and  $\{\varphi_n(z)\}_{n=0}^{\infty}$ the associated sequence of orthonormal polynomials associated. Then
\begin{enumerate}
\item Polynomials are dense in $L^{2}(\mathbb{T})$.
\item $\sum_{k=0}^{\infty} \vert \varphi_k(0) \vert^2 = \infty$.
\end{enumerate}
\end{corollary}

\begin{proof} The result is a consequence of Theorem \ref{teorema1} and Proposition \ref{proposicion4} since
$$
\gamma(\mathbf{T})=\dfrac{1}{\sum_{n=0}^{\infty} \vert \varphi_n(0) \vert^2},
$$
\noindent where $\gamma(\mathbf{T})=0$ whenever $\sum_{n=0}^{\infty} \vert \varphi_n(0) \vert^2=\infty$.
\end{proof}
\section{Bounded point evaluations from the matrix algebra point of view. Thomson's theorem revisited. }
\noindent We first recall the definitions of bounded point evaluation. Let $\mu$ be a non trivial positive measure with support on $\CC$. Recall ( see e.g. \cite{conway2} ) that
a point $z_0\in \CC$ is a  {\it bounded point evaluation} (in short,  {\it bpe}) for $P^{2}(\mu)$ if there exists a constant $C>0$ such that for every polynomial $p(z)$
$$
\vert p(z_0) \vert \leq C
\left(\int \vert p(z) \vert^2 d\mu\right)^{1/2}.
$$
\noindent Moreover, the
 point $z_{0}\in \CC$ is an {\it analytic bounded point evaluation} (in short an {\it abpe}) if there exists a constant $C>0$ and $\epsilon>0$ such that
for every $w\in \CC$ with $\vert w - z_0 \vert < \epsilon $ and for every polynomial $p(z)$ it holds
$$
\vert
 p(w) \vert^2 \leq c\int \vert p(z) \vert^2 d\mu.
$$
\begin{remark} Of course, an analytic bounded point evaluation is a bounded point evaluation. The converse is not true; indeed, any atomic isolated point is a bounded point evaluation but it is not an analytic bounded point evaluation.
\end{remark}
\noindent It is well known that if a point $z_0\in {\it supp}(\mu)$  is an atomic point of $\mu$, that is $\mu(\{ z_0 \})>0$, then it is a bounded point evaluation for $P^{2}(\mu)$. We prove it for the sake of completeness
\begin{lem} Let $z_0$ be an atomic point of a measure $\mu$ with $\mu(\{z_0 \})=\alpha>0$. Then $z_0$ is a bounded point evaluation for $P^{2}(\mu)$ with constant $C=\alpha^{-1/2}$.
\end{lem}
\begin{proof} Let $p(z)$ be a polynomial. Then,
$$
\vert p(z_0) \vert^2\mu(\{z_0\})= \int_{\{z_0\}}\vert p(z) \vert^2d\mu \leq \int \vert p(z) \vert^2d\mu.
$$
\noindent Therefore,
$$
\vert p(z_0) \vert \leq \dfrac{1}{\alpha^{-1/2}}\left(\int \vert p(z)\vert^2 d\mu\right).
$$
\end{proof}

\begin{remark}  It is important to point out the following remark in order to avoid confusions. In several references (see e.g. \cite{Thomson}, \cite{Brennan}) as a consequence of Thomson's theorem a dichotomy is established in the following way: let $\mu$ be a compactly supported measure in $\CC$: then either $P^{2}(\mu)=L^{2}(\mu)$ or there exist bounded point evaluations. This is not a dichotomy in the strict sense of the word, that is, in the sense that if one of them is true then the other must not happen. Indeed,  there are examples of compactly supported measures $\mu$ such that $P^{2}(\mu)=L^{2}(\mu)$ and nevertheless there exist bounded point evaluations.  Consider the example in \cite{EGT1} of any measure with $0\in {\it supp}(\mu)$ being a point mass and such that ${\it supp}(\mu)$ is a compact set with empty interior and with $K^{c}$ a connected set. In  particular, consider the sequence of infinite points $z_n=e^{\frac{2i\pi}{n}}$ in the unit circle, for example with weights $p_{n}=1/2^{n}$, $n \geq 1$.
In this case by Mergelyan theorem (see e.g. \cite{Gaier})  $P^{2}(\mu)=L^{2}(\mu)$ and nevertheless there exists a bounded point evaluation. Indeed, every $z_n$ is a bounded point evaluation for $P^{2}(\mu)$ since  it has been proved in Lemma 1. We give the formulation of Thomson's theorem as appears in (see e.g. \cite{Thomson}, \cite{Brennan}):
\end{remark}

\begin{theorem}  Let $\mu$ be a compactly supported measure in $\CC$. If $P^{2}(\mu) \neq L^{2}(\mu)$, then there exist a bounded point evaluation for $P^{2}(\mu)$.
\end{theorem}

\noindent Our aim in this section is to prove this theorem for  measures supported in  Jordan curves but with the novelty of using   techniques from the matrix algebra and using infinite HPD matrices. In order to do it, we first give a new approach  of  bounded point evaluations for a measure, and more generally, for  infinite HPD matrices.

\begin{definition} Let $\mathbf{M}$ be an HPD matrix
 and let $P^{2}(\mathbf{M})$ the closure of the polynomials with the inner product  induced by $\mathbf{M}$. Let $z_0\in \CC$, we say that  $z_{0}$ is a {\it bounded point evaluation for $P^{2}(\mathbf{M})$} if there exists a constant $C>0$ such that for every polynomial $p(z)$ it holds
$$
\vert p(z_{0}) \vert \leq C \Vert p(z) \Vert_{P^{2}(\mathbf{M})
}.
$$
\end{definition}

\begin{remark} Obviously, in the case of  $\mathbf{M}$ being a moment matrix associated with a measure $\mu$  the notion of
bounded point evaluation for $P^{2}(\mathbf{M})$ coincides with the usual of bounded point evaluation for $P^{2}(\mu)$.

\end{remark}

\noindent We need to introduce a new index for a given $z_0\in \CC$:

\begin{definition} Let $\mathbf{M}$ an HPD matrix and $z_{0}\in \CC$ and $k_{z_{0}}=\{z_0^k\}_{k=0}^{\infty}$, we define
$$
\gamma_{z_{0}}(\mathbf{M})= \inf \{v\mathbf{M}v^{*} :  \sum_{k=0}^{\infty} v_kz_0^k=1, \;\; v\in c_{00}\}.
$$
\end{definition}
\begin{remark} Note that $\gamma_{z_{0}}(\mathbf{M})\geq 0$ for every $z_{0}\in \CC$ and in the particular case that  $z_{0}=0$ then $\gamma_{z_{0}}(\mathbf{M})=\gamma(\mathbf{M})$.
\end{remark}

\noindent Next we  prove:
\begin{lem} \label{lema5}
Let  $\mathbf{M}$ be an HPD matrix. Then the following statements are equivalent:
\begin{enumerate}
\item $z_{0}$ is a bounded point evaluation for $P^{2}(\mathbf{M})$.
\item $\gamma_{z_{0}}(\mathbf{M})>0$.
\end{enumerate}

\end{lem}

\begin{proof} Assume first that $z_0$ is a bounded point evaluation of $P^{2}(\mathbf{M})$ with constant $C$, then for every $(v_0,v_1,\dots,v_n,0,0,\dots)$ and $p(z)=v_0+v_1z+\dots v_nz^n$
$$
\vert v k_{z_{0}} \vert = \vert v_0+v_1z_0+\dots +v_nz_0^n \vert \leq C \Vert p(z) \Vert_{P^{2}(\mathbf{M})}=C\Vert v \Vert^2_{\mathbf{M}}.
$$
\noindent In particular, if  $v k_{z_0}=1$ it holds $\Vert v \Vert_{\mathbf{M}} \geq \dfrac{1}{\sqrt{C}}$ and consequently $\gamma_{z_0}(\mathbf{M})\geq \dfrac{1}{\sqrt{C}}>0$.
\noindent On the other hand, if $\gamma_{z_0}(\mathbf{M})>0$ and $p(z)=\sum_{k=0}^{n} v_k z^k$ with $(v_{0},\dots, v_n,0,\dots)\in c_{00}$, either $p(z_0)=0$ and obviously $p(z_0)\leq \Vert p(z)\Vert_{P^{2}(\mathbf{M})}^2$, or $p(z_0)\neq 0$ and the vector $w=(w_{i})_{i=0}\in c_{00}$ defined by $w_i=\dfrac{v_i}{p(z_0)}$ for each $i\geq 0$ verifies $wk_{z_0}=1$ and consequently
$$
\gamma_{z_0}(\mathbf{M}) \leq w\mathbf{M}w^{*} =\dfrac{1}{\vert p(z_0)\vert^2} v\mathbf{M}v^{*}= \dfrac{1}{\vert p(z_0)\vert^2}\Vert p(z) \Vert^2_{\mathbf{M}}.
$$
\noindent Therefore,
$$
\vert p(z_0) \vert^2 \leq \dfrac{1}{\gamma_{z_0}(\mathbf{M})}\Vert p(z) \Vert^{2}_{P^{2}(\mathbf{M})}.
$$
\noindent and $z_0$ is a bounded point evaluation for $P^{2}(\mathbf{M})$ with constant $\dfrac{1}{(\gamma_{z_{0}}(\mathbf{M}))^{1/2}}$.
\end{proof}
\begin{remark} Note that  proof of the above Lemma gives us information about the constant of the bounded point evaluation; indeed,  $z_0$ is  bounded point evaluation for $P^{2}(\mathbf{M})$ with constant $\gamma_{z_0}(\mathbf{M})$.
\end{remark}
\noindent We may generalize the notion of kernels in the context of infinite HPD matrices. More precisely, for an infinite HPD matrix $\mathbf{M}$ we may define the associated kernels:
$$
K_{\mathbf{M}}(z,w)=\sum_{n=0}^{\infty} \varphi_n(z)\overline{\varphi_n(w)}
$$
\noindent for every $z,w$ such that the series converges. In this context, the extremal property for polynomials can be reformulated as:
\begin{lem} Let  $\mathbf{M}$ be an HPD matrix and let $\{\varphi_n(z)\}_{n=0}^{\infty}$ the sequence of orthonormal polynomials associated with $\mathbf{M}$. Then, the following are equivalent:
\begin{enumerate}

\item $\gamma_{z_0}(\mathbf{M})>0$.
\item $K_{\mathbf{M}}(z_0,z_0)=\sum_{n=0}^{\infty} \vert \varphi_n(z_0)\vert^2 <\infty$

\end{enumerate}
\end{lem}

\begin{proof} Using the notation for polynomials we may rewrite:
$$
\gamma_{z_0}(\mathbf{M})=\inf \{ \Vert p(z) \Vert^2_{P^{2}(\mathbf{M})} : p(z)\in \mathbb{P}[z], p(z_0)=1\}.
$$
\noindent First consider a polynomial $q(z)=\sum_{k=0}^{n}v_kz^k$ and we express  it in terms of the orthonormal basis, that is,
$q(z)=\sum_{k=0}^{n} w_k\varphi_k(z)$, then by using the Cauchy-Schwartz inequality
$$
1=\vert q(z_0) \vert \leq   \left(\sum_{k=0}^{n} \vert w_k \vert^2\right)^{1/2}\left(\sum_{k=0}^{n} \vert \varphi_k(z_0) \vert^2\right)^{1/2}.
$$
\noindent Then,
$$
\dfrac{1}{\sum_{k=0}^{n} \vert \varphi_k(z_0) \vert^2} \leq
  \sum_{k=0}^{n} \vert w_k \vert^2 = \Vert q(z) \Vert^2_{P^{2}(\mathbf{M})}.
$$
\noindent And by taking the infimum all over the polynomials of degree $n$,
$$
\dfrac{1}{\sum_{k=0}^{n} \vert \varphi_k(z_0) \vert^2} \leq \inf \{ \Vert q(z) \Vert^2, q(z)\in \mathbb{P}_n[z], q(z_0)=1\}.
$$
\noindent On the other hand, if we consider the polynomial $q(z)=K_{\mathbf{M}}(z,z_0)=\sum_{k=0}^{n} \varphi_k(z)\overline{\varphi_k(z_0)}$  the above infimum is  reached at this polynomial since
$$
\Vert q(z) \Vert^{2}_{P^{2}(\mathbf{M})}= \dfrac{1}{\sum_{k=0}^{n} \vert \varphi_k(z_0) \vert^2}.
$$
\noindent Then, for every $n\in \NN$,
\begin{eqnarray*}
\gamma_{z_0}(\mathbf{M})& = & \inf_{n} \min \{ \Vert q(z) \Vert^2, q(z)\in \mathbb{P}_n[z], q(z_0)=1\}\\
\mbox{} & = & \inf_{n} \dfrac{1}{\sum_{k=0}^{n} \vert \varphi_k(z_0) \vert^2}=\dfrac{1}{\sum_{k=0}^{\infty} \vert \varphi_k(z_0) \vert^2}.
\end{eqnarray*}
\end{proof}

\noindent We summarize all the equivalent notions of bounded point evaluations for an HPD matrix in the following proposition:

\begin{corollary} Let  $\mathbf{M}$ be an HPD matrix and let $\{\varphi_n(z)\}_{n=0}^{\infty}$ the sequence of orthonormal polynomials associated with $\mathbf{M}$. Then, the following are equivalent:

\begin{enumerate}
\item $z_{0}$ is a bounded point evaluation of $P^{2}(\mathbf{M})$.
\item $K_{\mathbf{M}}(z_0,z_0)=\sum_{n=0}^{\infty} \vert \varphi_n(z_0)\vert^2 <\infty$
\item $\gamma_{z_0}(\mathbf{M})>0$.
\item $\displaystyle{
\dfrac{1}{\gamma_{z_0}(\mathbf{M})}=\lim_{n \to \infty} (1,z_0,z_0^2,\dots, z_{0}^n)\mathbf{M}_{
n}^{-1}\left(
\begin{array}{c}
1 \\
z_0\\
\vdots \\
z_{0}^{n} \\
\end{array}
\right) >0}$.
\end{enumerate}
\end{corollary}

\begin{proof} We only need to  prove $(2)\Longleftrightarrow (4)$. It is well-known the expression
of the $n$-kernel by determinants
\[ K_{n}(y,z) =-\frac{1}{\Delta_{n}} \left |
\begin{array}{ccccc}
c_{00} & c_{10} & \ldots & c_{n,0} & 1 \\
c_{01} & c_{11} & \ldots & c_{n,1} & \overline{z} \\
\vdots & \vdots & \ddots & \vdots & \vdots \\
c_{0,n} & c_{1,n} & \ldots & c_{n,n} & \overline{z}^{n} \\
1 & y & \ldots & y^{n} & 0
\end{array}
\right | = (1,y, \ldots,y^{n} )
M_{n}^{-1} \left (
\begin{array}{c}
1 \\
\overline{z}\\
\vdots \\
\overline{z}^{n}
\end{array}
\right ), \]
the last identity is the Schur complement, that says
$\displaystyle{\left |\begin{array}{cc}
A & b \\
c^{t} & 0 \\
\end{array} \right |=|A| \langle c,A^{-1}b \rangle}$.

\end{proof}

\bigskip

\noindent As a consequence of this result we obtain our main result which is the following proof of Thomson's theorem for measures supported in Jordan curves via an algebraical way. This let us to provide an algebraical characterization of density of polynomials in terms of an index of the moment matrix associated with the measure:

\begin{theorem} Let $\Gamma$ be a Jordan curve such that $z_0 \in \interior \Gamma$ and let $\mu$ be a measure with infinite support in $\Gamma$  with associated moment matrix $\mathbf{M}$. Then, the following statements are equivalent
\begin{enumerate}
\item $\gamma_{z_0}(\mathbf{M})>0$.
\item $P^{2}(\mu) \neq L^{2}(\mu)$.
\item $z_0$ is a bounded point evaluation of $P^{2}(\mathbf{M})$.
\end{enumerate}
\end{theorem}

\begin{proof}
%%PRUEBA RAQUEL
%\noindent Since $z_0 \notin \sop(\mu)$, one can consider the transformation of the measure
%$
%d\widetilde{\mu}=d\mu(z-z_0)$. Denote by $\widetilde{M}$ the moment matrix associated with this measure $\widetilde{\mu}$ that will be a Jordan curve $\widetilde{\Gamma}$ with $0 \in int(\widetilde{\Gamma})$. Then, we may prove:
%\begin{equation} \label{igualdad}
%\gamma(\widetilde{\mathbf{M}}_{z_0})=\gamma_{z_0}(\mathbf{M})
%\end{equation}
%\noindent  We prove the above equality; in order to do it, we use that for the
%$$
%\int_{\widetilde{\Gamma}} \vert p(z)\vert^2d\widetilde{\mu} = \int_{\Gamma} \vert p(z-z_0)\vert^2d\mu
%$$
%\noindent Then, since for every $v_1,v_2,\dots,v_n\in \CC$
%$$
% \int \vert 1+v_1z+\dots+v_nz^n\vert^2d\widetilde{\mu} =\int \vert 1+v_1(z-z_{0})+\dots +v_n(z-z_0)^n\vert^2d\mu
%$$
%\noindent Passing to the infimum in $v_1,\dots,v_n\in \CC$, and $n\in \NN$, by ?? we have that
%\begin{eqnarray*}
%\gamma(\widetilde{\mathbf{M}}) &=&\inf \{ \int \vert 1+v_1z+\dots+v_n z^n\vert^2 d \widetilde{\mu} : v_1,\dots,v_n\in \CC, n\in \NN \} \\
%\mbox{} & = & \inf \{\int \vert 1+v_1(z-z_{0})+\dots +v_n(z-z_0)^n\vert^2d\mu :  v_1,\dots,v_n\in \CC, n\in \NN \}\\
%\mbox{} & = & \gamma_{z_0}(\mathbf{M}). \qquad \qquad \qquad
%\end{eqnarray*}
%\noindent Obviously, $P^{2}(\mu)=L^{2}(\mu)$ if and only if $L^{2}(\widetilde{\mu})=P^{2}(\widetilde{\mu})$ and consequently the result is a consequence of (\ref{igualdad}). %FIN PRUEBA RAQUEL

Let  $ \widetilde{\mathbf{M}}$ be the  moment matrix associated with image  measure $\tilde{\mu}$ obtained  after   a similarity map, $\varphi(z)=\alpha z+\beta$ onto $\CC$, is applied to the measure ${\mu}$ and $\tilde{\Gamma}$ the image Jordan curve. 
%obtained  after   a similarity map is applied on the support of $\mu$ and let $ \widetilde{\mathbf{M}_n}$ be the finite section of such matrix. Obviously $ \widetilde{\mathbf{M}}$ is the moment matrix associated with the image
%

\noindent We first prove that
\[ \gamma(\widetilde{\mathbf{M}})=\gamma_{z_{0}}(\mathbf{M}).\]
 It is know ( see e.g. \cite{jcam2007}) that the expression that relates the matrices $\mathbf{M}_{n}$ and $ \widetilde{\mathbf{M}}_{n}$, it is given by
\[ \widetilde{M}_{n}= \mathbf{A}_{n}(\alpha,\beta) \mathbf{M}_{n} \mathbf{A}^{*}_{n}(\alpha,\beta),\]
where  $\mathbf{A}_{n}(\alpha,\beta)$ is defined as in \cite{jcam2007},

\[ \mathbf{A}_{n}(\alpha,\beta) =
\left (
\begin{array}{ccccc}
{\binom 0 0} \alpha^{0} \beta^{0} & {\binom 1 0} \alpha^{0}
\beta^{1} & {\binom 2 0} \alpha^{0} \beta^{2} & {\binom 3 0}
\alpha^{0} \beta^{3} &
\ldots \\
0 & {\binom 1 1} \alpha^{1} \beta^{0} & {\binom 2 1} \alpha^{1}
\beta^{1} & {\binom 3 1} \alpha^{1} \beta^{2} &
\ldots \\
0 & 0 & {\binom 2 2} \alpha^{2} \beta^{0} & {\binom 3 2}
\alpha^{2} \beta^{1} &
\ldots \\
0 & 0 & 0 & {\binom 3 3} \alpha^{3} \beta^{0} &
\ldots \\
\vdots & \vdots & \vdots & \vdots & \ddots
\end{array}
\right ).\]
Note that if we choose $\alpha=1$ and a translation $\beta=-z_{0}$, then  $0 \in {\it int}(\tilde{\Gamma})$, we obtain:
\[ \widetilde{\mathbf{M}_{n}}=\mathbf{A}_{n}(1,-z_{0}) \; \mathbf{M}_{n} \; \mathbf{A}^{*}_{n}(1,-z_{0}). \] Since $\mathbf{M}$
is HPD, all its sections are invertible and we can write
\[ \widetilde{\mathbf{M}}_{n}^{-1}= [\mathbf{A}_{n}^{*}]^{-1}(1,-z_{0}) \; \mathbf{M}_{n}^{-1}\; \mathbf{A}_{n}^{-1}(1,-z_{0}) .\]
It is clear that
\begin{multline*}
\mathbf{A}_{n}(1,-z_{0})=
\left (
\begin{array}{cccccc}
1 &-z_{0} & z^2_{0} & -z_{0}^{3} & \ldots & \pm \binom{n}{n} z_{0}^{n} \\
0 & 1& -2 z_{0} &3 z_{0}^{2} & \ldots & \mp \binom{n}{n-1} z_{0}^{n-1} \\
0 & 0 & 1 & -3z_{0} & \ldots & \pm  \binom{n}{n-2} z_{0}^{-1} \\
\vdots & \vdots & \vdots & \vdots & \mbox{} & \vdots \\
0 & 0 & 0 & 0 & \ldots & \binom{n}{0} z^{0}
\end{array}
\right ) \quad \Rightarrow \quad \\
 \mathbf{A}_{n}^{-1}(1,-z_{0}) =\left (
\begin{array}{cccccc}
1 &z_{0} & z^2_{0} & z_{0}^{3} & \ldots &  \binom{n}{n} z_{0}^{n} \\
0 & 1& 2 z_{0} &3 z_{0}^{2} & \ldots & \binom{n}{n-1} z_{0}^{n-1} \\
0 & 0 & 1 & 3z_{0} & \ldots &   \binom{n}{n-2} z_{0}^{-1} \\
\vdots & \vdots & \vdots & \vdots & \mbox{} & \vdots \\
0 & 0 & 0 & 0 & \ldots & \binom{n}{0} z^{0}
\end{array}
\right ).
\end{multline*}
It is immediate that
\[\gamma(\widetilde{\mathbf{M}})=\lim_{n} \frac{|}{e_{0}^{t} \widetilde{M}^{-1}_{n} e_{0}}\]
and
\[ e_{0}^{t} \widetilde{\mathbf{M}}^{-1}_{n} e_{0} =e_{0}^{t} \left ([\mathbf{A}_{n}^{-1}]^{*}(1,-z_{0}) \; \mathbf{M}_{n}^{-1}\; \mathbf{A}_{n}^{-1}(1,-z_{0})\right  ) e_{0}.\]
We have also that
\[ e_{0}^{t} \mathbf{A}_{n}^{-1}(1,-z_{0})=(1,z_{0},z_{0}^2, \ldots, z_{0}^{n}),\qquad
[\mathbf{A}_n^{*}]^{-1}(1,-z_{0})e_{0}=(1,\overline{z}_{0},\overline{z}_{0}^2, \ldots, \overline{z}_{0}^{n})^{t}.\]
In consequence
\[ \gamma(\widetilde{\mathbf{M}})=\gamma_{z_{0}}(\mathbf{M}).\]

\noindent Now the result is consequence of Theorem $1$.
\end{proof}

\bigskip
\bigskip

\noindent We finish with some applications to our results:
\begin{corollary} Let $\Gamma$ be an analytic Jordan curve with non empty interior and let   $\mu$ be a measure with support in $\Gamma$ with weight function $w(z)$ defined on $\Gamma$ positive and continous and. Then $L^{2}(\mu)\neq P^{2}(\mu)$.

\end{corollary}

\begin{proof}
 Let  $z_0$ be an arbitrary interior point of $\Gamma$, by using  results in \cite{szego} it follows that
$K_{\mathbf{M}}(z_0,z_0)<\infty$. Therefore $\gamma_{z_{0}}(\mathbf{M})>0$, and by using Theorem $3$ we may conclude
 $L^{2}(\mu)\neq P^{2}(\mu)$.
\end{proof}

\noindent And also from Theorem $3$ the following result is obvious:

\begin{corollary}   Let $\Gamma$ be a Jordan curve and  $\mu$ be a measure with support in $\Gamma$. Assume that
$L^{2}(\mu)\neq P^{2}(\mu)$, then  every  $z_0 \in  \interior \Gamma $ is a bounded point evaluation of $\mu$.
\end{corollary}

\end{document}